\documentclass[11pt,a4paper,reqno]{amsart}
\usepackage{clrscode}
\usepackage{mathrsfs,amssymb}

\newcommand{\N}{\mathbb{N}}
\newcommand{\ol}[1]{\overline{#1}}

\newtheorem{theorem}{Theorem}

\newtheorem{corollary}{Corollary}

\newtheorem{proposition}{Proposition}
\newtheorem{lemma}{Lemma}

\newcounter{fiddletheoremtemp}

\begin{document}

\title[Small Overlap Monoids II]{Small Overlap Monoids II: Automatic Structures and Normal Forms}


\maketitle

\begin{center}

    MARK KAMBITES

    \medskip

    School of Mathematics, \ University of Manchester, \\
    Manchester M13 9PL, \ England.

\end{center}

\begin{abstract}
We show that any finite monoid or semigroup presentation satisfying the small
overlap condition $C(4)$ has word problem which is a deterministic
rational relation. It follows that the set of lexicographically minimal
words forms a regular language of normal forms, and that these normal forms
can be computed in linear time. We also deduce that $C(4)$ monoids and
semigroups are rational (in the sense of Sakarovitch), asynchronous automatic,
and word hyperbolic (in the sense of Duncan and Gilman). From this it follows
that $C(4)$ monoids satisfy analogues of Kleene's theorem, and admit
decision algorithms for the rational subset and finitely generated submonoid
membership problems. We also prove some automata-theoretic results which
may be of independent interest.
\end{abstract}

\section{Introduction}

Small overlap conditions are natural combinatorial conditions on monoid
and semigroup presentations, which serve to limit the complexity of derivation
sequences between equivalent words. They are the natural semigroup-theoretic
analogues of the \textit{small cancellation conditions} extensively employed
in combinatorial and geometric group theory \cite{Lyndon77}. It has long
been known that monoids with presentations satisfying the condition $C(3)$ have decidable
word problem \cite{Higgins92,Remmers71,Remmers80}; recent research of the author
\cite{K_smallover1} has
shown that the slightly stronger condition $C(4)$ implies that the word
problem is solvable in linear time on a 2-tape Turing machine.

In this paper, we take an automatic-theoretic approach to the study of
small overlap semigroups and monoids. Our main result is that the word
problem for any $C(4)$ monoid or semigroup presentation is a deterministic
rational relation (and moreover, effectively computable as such).
It follows from results of automata theory \cite{Johnson85,Johnson86}
that the set of all words which are lexicographically minimal in their equivalence
classes forms a regular language of normal forms, and that a normal form for
any element can be computed in linear time. We are also able to deduce that
every monoid or semigroup admitting a presentation satisfying the condition
$C(4)$ is rational (in the sense of Sakarovitch \cite{Sakarovitch87}) and
hence also asynchronous automatic, and word hyperbolic (in the sense of
Duncan and Gilman \cite{Duncan04}).
Another consequence is that $C(4)$ monoids satisfy an analogue
of Kleene's theorem (see for example \cite{Hopcroft69}): their rational
subsets coincide with their recognisable subsets. It follows also that
membership is uniformly decidable for rational subsets, and hence also
for finitely generated submonoids, of such monoids.

In addition to this introduction, this article comprises four sections.
Section~\ref{sec_background} briefly reviews the definitions of monoid
and semigroup presentations, and of small overlap conditions.
Section~\ref{sec_prefmod} contains some purely automata-theoretic
results which will be used to establish our main results, and may be of
some independent interest. In Section~\ref{sec_smalloverlap} we combine
the results of the previous section with those of \cite{K_smallover1}
to prove our main theorem. Finally, in Section~\ref{sec_consequences} we
deduce some consequences.

\section{Preliminaries}\label{sec_background}

In this section we briefly recall the key definitions of semigroup and monoid
presentations and of small overlap conditions, which will be used in the rest
of this paper.

Let $A$ be a finite alphabet (set of symbols). A \textit{word} over $A$
is a finite sequence of zero or more elements from $A$.
The set of all words
over $A$ is denoted $A^*$; under the operation of \textit{concatenation}
it forms a monoid, called the \textit{free monoid} on $A$. The length of
a word $w \in A^*$ is denoted $|w|$. The unique \textit{empty word} of length
$0$ is denoted $\epsilon$; it forms the identity element of the monoid
$A^*$. The set $A^*\setminus \lbrace \epsilon \rbrace$ of non-empty
words forms a subsemigroup of $A^*$, called the \textit{free semigroup
on $A$} and denoted $A^+$. For $k \in \mathbb{N}$ we
write $A^k$, $A^{\leq k}$ and $A^{< k}$ to denote the set of words in $A^*$
of length respectively exactly $k$, less than or equal to $k$, and strictly
less than $k$. If $w \in A^*$ is a word, we write $w^R$ to denote the
\textit{reverse} of $w$, that is, the word composed of the letters of $w$
written in reverse order.

A finite \textit{monoid presentation} $\langle A \mid R \rangle$ consists of
a finite alphabet $A$ (the letters of which are called \textit{generators}),
together with a finite set $R \subseteq A^* \times A^*$ of
pairs of words (called \textit{relations}). We say that $u, v \in A^*$ are
\textit{one-step equivalent} if $u = axb$ and $v = ayb$ for some
possibly empty words $a, b \in A^*$ and relation $(x,y) \in R$ or $(y,x) \in R$.
We say that $u$ and $v$ are \textit{equivalent}, and write $u \equiv_R v$
or just $u \equiv v$, if there is a finite sequence of words beginning
with $u$ and ending with $v$, each term of which but the last is one-step
equivalent to its successor. Equivalence is clearly an equivalence relation;
in fact it is the least equivalence relation containing $R$ and compatible
with the multiplication in $A^*$. We write $\ol{u}$ for the equivalence class
of a word $u \in A^*$. The equivalence classes form a monoid
with multiplication well-defined by $\ol{u} \ \ol{v} = \ol{uv}$;
this is called the \textit{monoid presented} by the presentation.

The \textit{word problem} for a (fixed) monoid presentation
$\langle A \mid R \rangle$ is the algorithmic problem of, given as
input two words $u, v \in A^*$, deciding whether $u \equiv_R v$.

Definitions corresponding to all of those above can also be made for
semigroups (without necessarily an identity element), by taking $A^+$ in
place of $A^*$ (in all places except the definition of one-step equivalence,
where $a$ and $b$ must still be allowed to be empty).

Now suppose we have a fixed monoid or semigroup presentation
$\langle A \mid R \rangle$. We begin by recalling some basic definitions
from the theory of small overlap conditions \cite{Higgins92,Remmers71}.
A \textit{relation word} is a word which appears as one side of a relation
in $R$. A \textit{piece} is a word which appears more than once as a factor
in the relations, either as a factor of two different relation words, or as
a factor of the same relation word in two different (but possibly overlapping)
places. Let $m \in \mathbb{N}$ be a positive integer. The presentation is said to 
\textit{satisfy $C(m)$} if no relation word can be written as a product of
\textit{strictly fewer than $m$} pieces. Thus $C(1)$ says that no relation
word is empty (which in the semigroup case is a trivial requirement);
$C(2)$ says that no relation word is a factor of another.

Retaining our fixed presentation, we now recall some more specialist
terminology from \cite{K_smallover1}. For each relation word $R$, let $X_R$ and $Z_R$ denote respectively the
longest prefix of $R$ which is a piece, and the longest suffix of $R$
which is a piece. If the presentation satisfies $C(3)$ then $R$ cannot be
written as a product of two pieces, so this prefix and suffix cannot meet;
thus, $R$ admits a factorisation $X_R Y_R Z_R$ for some non-empty word $Y_R$.
If moreover the presentation satisfies the stronger condition $C(4)$ then $R$
cannot be written as a product of three pieces, so $Y_R$ is not a piece. The
converse also holds: a $C(3)$ presentation such that no $Y_R$ is a piece
is a $C(4)$ presentation. We
call $X_R$, $Y_R$ and $Z_R$ the \textit{maximal piece prefix}, the
\textit{middle word} and the \textit{maximal piece suffix} respectively
of $R$.

If $R$ is a relation word we write $\ol{R}$ for the (necessarily unique,
as a result of the small overlap condition)
word such that $(R, \ol{R})$ or $(\ol{R}, R)$ is a relation in the
presentation. We write
$\ol{X_R}$, $\ol{Y_R}$ and $\ol{Z_R}$ for $X_{\ol{R}}$, $Y_{\ol{R}}$ and
$Z_{\ol{R}}$ respectively. (This is an abuse of notation since, for example,
the word $X_R$ may be a maximal piece prefix of two distinct relation words, but
we shall be careful to ensure that the meaning is clear from the context.)

A \textit{relation prefix} of a word
is a prefix which admits a (necessarily unique, as a consequence of the
small overlap condition) factorisation of the form $a X Y$ where $X$ and $Y$
are the maximal piece prefix and middle word respectively of some relation
word $XYZ$. An \textit{overlap prefix (of length $n$)} of
a word $u$ is a relation prefix which admits an (again necessarily unique)
factorisation of the form $b X_1 Y_1' X_2 Y_2' \dots X_n Y_n$ where
\begin{itemize}
\item $n \geq 1$;
\item $b X_1 Y_1' X_2 Y_2' \dots X_n Y_n$ has no factor of the form $X_0Y_0$,
where $X_0$ and $Y_0$ are the maximal piece prefix and middle word respectively
of some relation word, beginning before the end of the prefix $b$;
\item for each $1 \leq i \leq n$, $R_i = X_i Y_i Z_i$ is a relation word with
$X_i$ and $Z_i$ the maximal piece prefix and suffix respectively; and
\item for each $1 \leq i < n$, $Y_i'$ is a proper, non-empty prefix of $Y_i$.
\end{itemize}

Let $u \in A^*$ be a word and let $p$ be a
piece. We say that $u$ is \textit{$p$-active} if $p u$ has a relation prefix
$a XY$ with $|a| < |p|$, and \textit{$p$-inactive} otherwise.

We now recall some basic definitions from automata theory.
If $A$ is an alphabet, we denote by $A^\$ $ the alphabet $A \cup \lbrace \$ \rbrace$
where $\$ $ is a new symbol not in $A$. The symbol $\$ $ will be used as an
\textit{end-marker} for certain types of automata. If $R \subseteq A_1^* \times A_2^*$
is a relation, we denote by $R^\$ $ the set
$$R^\$ \ = \ R \ (\$, \$) \ = \ \lbrace (u \$, v \$) \mid (u,v) \in R \rbrace \ \subseteq \ A_1^* \$ \times A_2^* \$ \ \subseteq \ (A_1^\$)^* \times (A_2^\$)^*.$$

A \textit{rational transducer} from an alphabet $A_1$ to an alphabet $A_2$ is a finite
directed graph with edges labelled by elements of $A_1^* \times A_2^*$, together
with a distinguished initial vertex and a set of distinguished terminal
vertices. The labelling of edges extends to a labelling of paths via the
multiplication in the direct product monoid $A_1^* \times A_2^*$. A pair
$(u,v) \in A_1^* \times A_2^*$ is \textit{accepted} by the transducer if it
labels some path from the initial vertex to a terminal vertex. The
\textit{relation accepted} by the transducer is the set of all pairs
accepted. A relation accepted by some transducer is called a
\textit{rational relation} or \textit{rational transduction}. Transductions,
which were introduced in \cite{Elgot65}, are of fundamental importance in
the theory of formal languages and automata; a detailed study can be found
in \cite{Berstel79}.

A \textit{deterministic 2-tape finite automaton} consists of 
two alphabets $A_1$ and $A_2$, a finite state set $Q$ partitioned into two
disjoint subsets $Q_1$ and $Q_2$ with a distinguished initial state and
set of distinguished terminal states, and for each $i = 1,2$ a partial
function
$$\delta_i : Q_i \times A_i^\$ \to Q.$$
Let $\mapsto$ be the smallest binary relation on $A_1^* \$ \times A_2^* \$ \times Q$
such that
\begin{itemize}
\item $(au,v, p) \mapsto (u, v, q)$ for all $a \in A_1$, $u \in A_1^* \$ $,
$v \in A_2^* \$ $, $p \in Q_1$, $q \in Q$ such that $\delta_1(p, a)$ is
defined and equal to $q$; and 
\item $(u,bv, p) \mapsto (u, v, q)$ for all $b \in A_2$, $u \in A_1^* \$ $,
$v \in A_2^* \$ $, $p \in Q_2$, $q \in Q$ such that $\delta_2(p, b)$ is
defined and equal to $q$;
\end{itemize}
and let $\mapsto^*$ be the reflexive, transitive closure of $\mapsto$.
We say that a pair $(u,v) \in A_1 \times A_2$ is \textit{accepted} by the
automaton if there exists
an initial state $q_0$ and a terminal state $q_1$ such that
that $(u \$, v \$, q_0) \mapsto^* (\epsilon, \epsilon, q_1)$. Once again,
the \textit{relation accepted} by the automaton is the set of all pairs
accepted.

A relation is called a \textit{deterministic rational relation} if it is
accepted by a deterministic 2-tape automaton, and a \textit{reverse
deterministic rational relation} if the relation
$$\lbrace (u^R, v^R) \mid (u, v) \in R \rbrace$$
is accepted by a deterministic 2-tape automaton. In general, a
deterministic rational relation need not be reverse deterministic
rational \cite[Theorem~1]{Fischer68}. Every [reverse] deterministic
rational relation is accepted by a transducer \cite{Fischer68} and so
is indeed a rational relation.
The following elementary proposition gives a partial converse to this
statement; the general
idea is well known but the
precise formulation we need does not seem to have appeared in the literature, so
for completeness we give an outline proof.
\begin{proposition}\label{prop_dettransducer_condition}
Let $R \subseteq A_1^* \times A_2^*$ be a relation and suppose $R^\$ $ is accepted
by a transducer with the property that for every state $q$, one of the following (mutually
exclusive) conditions holds:
\begin{itemize}
\item[(i)] $q$ has an edge leaving it, and every edge leaving $q$ has the form $(a, \epsilon)$ for some $a \in A_1^\$ $,
      and there is at most one such edge for each $a \in A_1^\$ $;
\item[(ii)] $q$ has an edge leaving it, and every edge leaving $q$ has the form $(\epsilon, a)$ for some $a \in A_2^\$ $,
      and there is at most one such edge for each $a \in A_2^\$ $;
\item[(iii)] there are no edges leaving $q$;
\item[(iv)] there is exactly one edge leaving $q$, and that edge has label
      $(\epsilon, \epsilon)$;
\end{itemize}
Then $R$ is accepted by a deterministic 2-tape automaton.
\end{proposition}
\begin{proof}
Let $M$ be the transducer accepting $R^\$ $ with the given property, and let
$Q$ be the state set of $M$.
 Notice that for each state $q$, there is at most one state, which
we call $\ol{q}$, with the property that there is a path from $q$
to $\ol{q}$ labelled $(\epsilon, \epsilon)$ and $\ol{q}$ satisfies condition
(i) or (ii) in the statement of the proposition. Since (i) and (ii) are
mutually exclusive, we may choose a partition $Q = Q_1 \cup Q_2$ of $Q$
into disjoint subsets such that for every $q \in Q$ with $\ol{q}$ defined
we have that 
$\ol{q}$ satisfies condition (i) if and only if $q \in Q_1$, and similarly
$\ol{q}$ satisfies condition (ii) if and only if $q \in Q_2$. (States $q$
for which $\ol{q}$ is not defined may be assigned arbitrarily to either $Q_1$
or $Q_2$).

We now define a new deterministic 2-tape automaton $N$ as follows. The two
tape alphabets of $N$ are $A_1$ and $A_2$. The state
set of $N$ is the state set $Q$ of $M$ partitioned into the subsets $Q_1$ and
$Q_2$ constructed above. The initial state of $N$ is the
initial state of $M$. The terminal states of $N$ consist of all states
$p \in Q$ such that $M$ has a path from $p$ to a terminal state with label
$(\epsilon, \epsilon)$.
For each $a \in A_1^\$ $, $p \in Q_1$ and $q \in Q$ we set $\delta_1(p, a) = q$
if and only if $\ol{p}$ is defined and $M$ has an edge from $\ol{p}$ to $q$ with label $(a, \epsilon)$.
Similarly, for each $a \in A_2^\$ $, $p \in Q_2$ and $q \in Q$ we set
$\delta_2(p, a) = q$
if and only if $\ol{p}$ is defined and $M$ has an edge from $\ol{p}$ to $q$ with label $(\epsilon, a)$.
It follows directly from the criteria on the automata that each $\delta_i$ is
a well-defined partial function from $Q_i \times A_i^\$ $ to $Q$.

It is now a routine matter to verify that the deterministic 2-tape automaton
$N$ accepts a pair $(u,v)$ if and only if $M$ accepts $(u \$, v \$)$.
\end{proof}

\section{Prefix-Rewriting Automata}\label{sec_prefmod}

In this section, we study a type of automaton called a
\textit{2-tape prefix-rewriting automaton}. We show that any relation
accepted by a [deterministic] 2-tape prefix-rewriting automaton
with a certain property called \textit{bounded expansion} is a
[deterministic]
rational relation. In Section~\ref{sec_smalloverlap} we shall apply this
result to show that the word problem for a $C(4)$ monoid presentation is
a deterministic rational relation.

Let $k \in \N$ and $A_1$ and $A_2$ be finite alphabets. A \textit{$k$-prefix-rewriting automaton
from $A_1$ to $A_2$} is a finite directed graph with edges labelled by elements of
$$\left( (A_1^{\leq k} \times A_1^{\leq k}) \cup (A_1^{< k} \$ \times A_1^{< k} \$) \right) \times \left( (A_2^{\leq k} \times A_2^{\leq k}) \cup (A_2^{< k} \$ \times A_2^{< k} \$) \right),$$
together with a distinguished initial vertex and a set of
distinguished terminal vertices. Given such an automaton with vertex set $Q$,
we define a binary relation $\to$ on $A_1^* \$ \times A_2^* \$ \times Q$ by
$$(u_1 \$, v_1 \$, q_1) \to (u_2 \$, v_2 \$, q_2)$$
if and only if there exist words $x_1$, $x_2$, $y_1$, $y_2$, $u'$ and $v'$
in the appropriate alphabets such that
$$u_1 = x_1 u', \ u_2 = x_2 u', \ v_1 = y_1 v', \ v_2 = y_2 v'$$
and $(x_1, x_2, y_1, y_2)$ labels an edge from $q_1$ to $q_2$. If this
holds we say that \textit{the edge $e$ is applicable in the configuration
$(u_1 \$, v_1 \$, q_1)$}.
We call the automaton \textit{deterministic} if in each configuration
$(u,v,q) \in A_1^* \$ \times A_2^* \$ \times Q$ there is at most one
edge applicable.

Let $\to^*$ denote the reflexive, transitive closure of the relation $\to$.
We say that a pair $(u,v) \in A_1^* \times A_2^*$ is \textit{accepted} by the 
automaton if there exists a terminal state $q_1$
such that
$$(u \$, v \$, q_0) \to^* (\$, \$, q_1)$$
where $q_0$ is the initial state. As usual, 
the \textit{relation accepted} by the automaton is the set of all pairs in
$A_1^* \times A_2^*$ which are accepted
by the automaton.

Intuitively, a 2-tape prefix-rewriting automaton is very similar to a
2-pushdown automaton; the only essential difference is that we allow both
stacks to be initialised with non-empty words, and view the automaton as
accepting pairs of words and defining a relation instead of a language.
As one might expect, such automata are extremely powerful, being easily seen
to accept in particular any relation of the form $L \times \lbrace \epsilon \rbrace$
where $L$ is a recursively enumerable
language. However, we shall
be interested in a more restricted class of such automata.
We say that a prefix-rewriting automaton has \textit{bounded expansion} if
there exists a constant $b \in \N$ such that whenever
$$(u_1, v_1, q_1) \to^* (u_2, v_2, q_2)$$ we have
$|u_2| \leq |u_1| + b$ and $|v_2| \leq |v_2| + b$. We call such a value of
$b$ an \textit{expansion bound} for the automaton.

Note that the bounded expansion condition places a requirement on the
contents of each store independently. This contrasts with the
\textit{shrinking} and \textit{length-reducing} conditions on 2-pushdown
automata, used to describe growing context-sensitive and Church-Rosser languages
\cite{Buntrock98}, where a restriction is applied to the total size
of the 2 stores considered together. It transpires that our condition
is a very strong one, in that a relation accepted by a prefix-rewriting
automaton with bounded expansion is necessarily rational.

\begin{theorem}\label{thm_bpm_imp_rational}
Any relation accepted by a [deterministic] 2-tape prefix-rewriting
automaton with bounded expansion is a [deterministic] rational
transduction. Moreover, given a [deterministic] 2-tape prefix-rewriting
automaton and an expansion bound for it, one can effectively construct a
[deterministic] transducer recognising the same relation.
\end{theorem}
\begin{proof}
Let $M$ be a $2$-tape $k$-prefix-rewriting automaton with bounded expansion accepting
a relation $R \subseteq A_1^* \times A_2^*$, and let $b \in \N$ be an
expansion bound for $M$. We construct from $M$ a finite transducer $N$ which simulates $M$
and so accepts $R^\$ $. Intuitively, the new transducer will read $u$ and
$v$, buffering at least the first $k$ characters of each in the finite state
control. Prefix-modification can thus be simulated by modifying only the
contents of the finite state control. Since a prefix-rewriting automaton
can replace a prefix with a longer one, it may be necessary to store more
than $k$ characters of each word in the finite state control, but the
expansion bound serves to ensure that a buffer of some fixed size (namely $k+b$) will always
suffice.

Formally, for $i = 1,2$ we let $C_i = A_i^{\leq k+b} \cup A_i^{< k+b} \$ $
and let $B_i$ be the set of all words $x \in C_i$ such that either
$|x| \geq k$ or the final letter of $x$ is $\$ $. (Intuitively, $C_i$ will be
the set of all possible states for the buffer on tape $i$, while $B_i$
will be the set of ``adequately populated'' buffer states in which it is not
immediately necessary to read any more of the input word.)

We construct a transducer $N$ as follows. The state set of $N$ is
$C_1 \times C_2 \times Q$ where $Q$ is the state set of $M$. The initial state is $(\epsilon, \epsilon, q_0)$
where $q_0$ is the initial state of $M$. The terminal states are those of
the form $(\$, \$, q)$ with $q$ a terminal state of $M$. The edges are
as follows:
\begin{itemize}
\item[(1)] for every $x \in C_1$, $y \in C_2$ with $x \notin B_1$, every
$a \in A_1^\$ $ such
that $xa \in C_1$ and every
state $q$, there is an edge from $(x,y,q)$ to $(xa,y,q)$ with label $(a, \epsilon)$;
\item[(2)] for every $x \in C_1$, $y \in C_2$ with $x \in B_1$ but $y \notin B_2$,
every $a \in A_2^\$ $ such that $ya \in C_2$ and every state $q$, there is an edge from
$(x,y,q)$ to $(x,ya,q)$ with label $(\epsilon, a)$;
\item[(3)] for each edge in $M$ from $p$ to $q$ with label
$(u_1, u_2, v_1, v_2)$ and each $x', y'$ such that $u_1 x' \in B_1$
and $v_1 y' \in B_2$, there is an edge
from $(u_1 x', v_1 y', p)$ to $(u_2 x', v_2 y', q)$ with label $(\epsilon, \epsilon)$
provided $u_2 x' \in C_1$ and $v_2 u' \in C_2$.
\end{itemize}
Edges of types (1) and (2) serve simply to read the input words into
the buffers until each contains sufficient data (at least $k$
letters or the entire of the input if this is less), while edges of type
(3) simulate the transitions of the prefix-rewriting
automaton $M$ by operating only on the buffers.

Notice that once the transducer reaches a state in
$A_1^{<k+b} \$ \times C_2 \times Q$ (that is, one where the first buffer
content contains the symbol $\$ $), it will always remain in such a state,
and will never again read from the first input word.
Similarly, once it reaches a state in $C_1 \times A_2^{<k+b} \$ \times Q$
it will always remain in such a state and will never again read from the
second input word. Noting also that all the terminal states lie in both of
these sets, it follows that all pairs accepted by the transducer lie in
$A_1^* \$ \times A_2^* \$ $.

We say that a configuration $(u_1,v_1,q_1)$ has expansion bound $(c, d) \in \mathbb{N} \times \mathbb{N}$
if whenever $(u_1, v_1, q_1) \to^* (u_2, v_2, q_2)$ we have $|u_2| \leq |u_1| + c$ and
$|v_2| \leq |u_1| + d$. Note that the expansion bound condition on the automaton
means that $(b,b)$ is an expansion bound for every configuration. We shall
need the following lemma.

\begin{lemma}\label{lemma_prefmodimptransducer}
Suppose $(u_1,v_1,q_1) \to^* (u_2,v_2,q_2)$ in the prefix-rewriting
automaton $M$. Suppose further than $(u_1, v_1, q_1)$ has expansion bound
$(c_1, d_1)$ and that
$u_1 = s_1 s_1'$, $v_1 = t_1 t_1'$ where
$|s_1| \leq k+b-c_1$ and $|t_1| \leq k+b-d_1$.
Then there exist factorisations $u_2 = s_2 s_2'$
and $v_2 = t_2 t_2'$ and an expansion bound $(c_2, d_2)$ for
$(u_2, v_2, q_2)$ such that $|s_2| \leq k+b-c_2$, $|t_2| \leq k+b-d_2$
and the transducer $N$ has a path from $(s_1, t_1, q_1)$ to $(s_2, t_2, q_2)$
with label $(g, h)$ where $s_1' = g s_2'$ and $t_1' = h t_2'$.
\end{lemma}
\begin{proof}
We use induction on the number of steps in the transition sequence from
from $(u_1, v_1, q_1)$ to $(u_2, v_2, q_2)$. Certainly if
$(u_1, v_1, q_1) = (u_2, v_2, q_2)$ it suffices to take
$s_2 = s_1$, $s_2' = s_1'$, $t_2 = t_1$, $t_2' = t_1'$,
$c_2 = c_1$, $d_2 = d_1$ and $g = h = \epsilon$.

Next we consider one-step case, that is, the case in which
$(u_1,v_1,q_1) \to (u_2,v_2,q_2)$.
Let $g$ be the shortest prefix of $s_1'$ such that
$s_1 g \in B_1$; similarly, let $h$ be the shortest prefix of $t_1'$
such that $t_1 h \in B_2$. It follows easily from the definition that our
transducer $N$ has a path from $(s_1, t_1, q_1)$ to $(s_1 g, t_1 h, q_1)$
with label $(g, h)$.

Now since $(u_1, v_1, q_1) \to (u_2, v_2, q_2)$,
by definition there exist words $x_1$, $x_2$, $y_1$, $y_2$, $u'$ and $v'$
such that
$u_1 = x_1 u'$, $u_2 = x_2 u'$, $v_1 = y_1 v'$, $v_2 = y_2 v'$ and
$(x_1, x_2, y_1, y_2)$ labels an edge from $q_1$ to $q_2$. Since
$|x_1|, |y_1| \leq k$
we have that $x_1$ and $y_1$ are prefixes of $s_1 g$ and $t_1 h$
respectively, say $s_1 g = x_1 x'$ and $t_1 h = y_1 y'$. But now
by the definition of our transducer, there is an edge
from $(s_1 g = x_1 x', t_1 h = y_1 y', q_1)$ to
$(x_2 x', y_2 y', q_2)$ with label $(\epsilon, \epsilon)$. Thus, setting
$s_2 = x_2 x'$ and $t_2 = y_2 y'$ and defining $s_2'$ and $t_2'$ accordingly,
we obtain a path from $(s_1, t_1, q_1)$ to $(s_2, t_2, q_2)$
with label $(g,h)$.

Now we have
$$x_2 x' s_2' = s_2 s_2' = u_2 = x_2 u'$$
so cancelling on the left we obtain $u' = x' s_2'$. But now
$$s_1 s_1' = u_1 = x_1 u' = x_1 x' s_2' = s_1 g s_2'$$
so cancelling again yields $s_1' = g s_2'$ as claimed. An entirely
similar argument shows that $t_1' = h t_2'$.

Next, notice that we have $|u_1| - |u_2| = |s_1| - |s_2|$ and similarly
$|v_1| - |v_2| = |s_1| - |s_2|$. Set $c_2 = c_1 + |s_1| - |s_2|$ and
$d_2 = d_1 + |t_1| - |t_2|$.
Clearly since any state derivable from $(u_2, v_2, q_2)$ is also derivable
from $(u_1, v_1, q_1)$, it is readily verified that $(c_2, d_2)$ is an
expansion bound for $(u_2, v_2, q_2)$. But now we have
$$|s_2| \ = \ |s_1| + c_1 - c_2 \ \leq \ (k + b - c_1) + c_1 - c_2 \ = \ k + b - c_2$$
and similarly $|t_2| \leq k + b - d_2$ as required to complete the proof of the
lemma in the one-step case.

The inductive argument for the general case is now straightforward.
\end{proof}

Now if $(u,v)$ is accepted by the prefix-rewriting automaton then
by definition we have $(u \$,v \$, q_0) \to^* (\$, \$, q_t)$ where
$q_0$ is the initial state and $q_t$ is some terminal state.
Since the automaton has
expansion bound $b$, the state $(u \$,v \$, q_0)$ has expansion bound $(b,b)$.
So taking $u_1 = u$, $v_1 = v$, $q_1 = q_0$, $q_2 = q_t$ $c_1 = d_1 = b$,
$s_1 = t_1 = \epsilon$, $s_1' = u$ and $s_2' = v$ and applying
Lemma~\ref{lemma_prefmodimptransducer}, our transducer
has a path from $(\epsilon, \epsilon, q_0)$ to 
$(s_2, t_2, q_t)$ with label $(g,h)$ where $s_2 s_2' = t_2 t_2' = \$ $,
$u = s_1' = g s_2'$ and $v = t_1' = h t_2'$.

Now either $s_2 = \epsilon$ and $s_2' = \$ $, or $s_2 = \$ $ and $s_2' = \epsilon$.
In the latter case we have $g = u \$ $. In the former case we have $g = u$
and there is clearly an edge from $(s_2, t_2, q_t)$ to
$(s_2 \$ = \$, t_2, q_t)$ labelled $(\$, \epsilon)$, so in either case there is
a path from $(\epsilon, \epsilon, q_0)$ to $(\$, t_2, q_t)$ with label
$(u \$, h)$. A similar argument deals with the case that $h = v$, showing that
in all cases there is a path from the start state
$(\epsilon, \epsilon, q_0)$ to the terminal state $(\$, \$, q_t)$ with
label $(u \$,v \$)$. Thus, the transducer $N$
accepts $(u \$,v \$)$ as required.

Conversely, suppose $(u \$, v \$)$ is accepted by our transducer. Then there
must be a path $\pi$ from $(\epsilon, \epsilon, q_0)$ to $(\$, \$, q_t)$ for some initial state
$q_0$ and terminal state $q_t$.  Now clearly $\pi$ admits a unique decomposition
of the form
$$\pi \ = \ \lambda_0 \rho_1 \lambda_1 \rho_2 \dots \rho_n \lambda_n$$
where each $\rho_i$ is a single edge of type (3) and each $\lambda_i$ is
a (possibly empty) path consisting entirely of edges of types (1) and (2).
Clearly each $\rho_i$ has label $(\epsilon, \epsilon)$. Suppose each
$\lambda_i$ has label $(u_i, v_i)$; then clearly
$u \$ = u_0 u_1 \dots u_n$ and $v \$ = v_0 v_1 \dots v_n$. Suppose that for
$0 \leq i \leq n$, after
traversing the initial segment of the path $\pi$ up to and including $\lambda_i$,
the automaton is in configuration $(x_i, y_i, q_i)$. Notice
that, since the paths $\lambda$ do not change the state component, $q_0$
is consistent with its use above, and in particular is an initial state
in the prefix-rewriting automaton $M$. Similarly, $q_n = q_t$ is a terminal
state of $M$.
Now for $0 \leq i \leq n$ define
$$c_i = x_i u_{i+1} u_{i+2} \dots u_n \text{ and } d_i = y_i v_{i+1} v_{i+2} \dots v_n.$$
Clearly we have that $x_0 = u_0$ and $y_0 = v_0$, from which it follows that
$c_0 = u \$ $ and $d_0 = v \$ $. We also have $x_n = y_n = \$ $ so that
$c_n = d_n = \$ $.

Now it is straightforward to see that for $1 \leq i \leq n$ we have
$$(c_{i-1}, d_{i-1}, q_{i-1}) \to (c_i, d_i, q_i)$$
so that
$$(u \$, v \$,q_0) = (c_0, d_0, q_0) \ \to^* \ (c_n, d_n, q_n) = (\$, \$, q_t).$$
which by definition means that $(u,v)$ is accepted by the 2-tape
prefix-rewriting automaton $M$.
This completes the proof that the transducer $N$ accepts the relation
$R^\$ $. It is easy to show that for any relation $T$, $T$
is a rational relation if and only if $T^\$ $ is a rational relation,
so this suffices to prove that $R$ is a rational relation.

Finally, suppose that the original prefix-rewriting automaton $M$ is
deterministic. We claim that the transducer $N$ which we have constructed
to accept $R^\$ $ satisfies the conditions of Proposition~\ref{prop_dettransducer_condition}, from which
it will follow that $R$ is a deterministic rational relation, as required.

To this end, consider a state $(x,y,q)$ in $N$. If $x \notin B_1$ then it
follows immediately from the definition that all out-edges have labels of
the form $(a, \epsilon)$ with $a \in A_1$ and that there is exactly one such
for each $a \in A$, so that condition (i) holds. Similarly, if $x \in B_1$
but $y \notin B_2$ then all out-edges have labels of the form $(\epsilon, a)$
and there is exactly one such for each $a \in A_2$ so condition (ii) holds.

Finally, suppose $x \in B_1$ and $y \in B_2$. From the definition of $N$,
any edge leaving $(x,y,p)$ must have label $(\epsilon, \epsilon)$. If there
were more than one such edge, then each would correspond to a different possible
transition in $M$ from the state $(x,y,p)$; but by the determinism assumption
on $M$ there can only be one such transition, so this would give a
contradiction. Thus we deduce that there is at most one such edge, so that
either condition (iii) or condition (iv) holds. This completes the proof.
\end{proof}
We emphasise that Theorem~\ref{thm_bpm_imp_rational} does \textit{not} give
a means to effectively construct a transducer for a relation $R$ starting
only from a 2-tape prefix-rewriting automaton with bounded expansion which
accept $R$. The construction in the proof makes explicit use of the
expansion bound for the prefix-rewriting automaton, and it is not clear
that one can effectively compute an expansion bound from the automaton,
even given the knowledge that such a bound exists.

\section{Automata for the Word Problem in Small Overlap Monoids}\label{sec_smalloverlap}

The aim of this section is to show that the word problem for any $C(4)$
monoid must be a deterministic rational relation.
Throughout this section, we fix a monoid presentation $\langle A \mid R \rangle$
satisfying the condition $C(4)$.

In \cite{K_smallover1} we presented an efficient recursive algorithm which can
be used to solve the word problem for such a presentation. For ease of reference
the algorithm is reproduced in Figure~1.
\begin{figure}
\begin{codebox}
\Procname{$\proc{WP-Prefix}(u, v, p)$}
\li     \If $u = \epsilon$ or $v = \epsilon$ \label{li_start_a}
\li         \Then \If $u = \epsilon$ and $v = \epsilon$ and $p = \epsilon$
\li             \Then \Return \const{Yes}                \label{li_allepsilon}
\li             \Else \Return \const{No}                 \label{li_someepsilon}
            \End \label{li_end_a}
\li     \ElseIf $u$ does not have the form $XYu'$ with $XY$ a clean overlap prefix
\li     \Then \If $u$ and $v$ begin with different letters \label{li_start_b}
\li         \Then \Return \const{No}                     \label{li_uvdifferentstart}
\li        \ElseIf $p \neq \epsilon$ and $u$ and $p$ begin with
different letters
\li         \Then \Return \const{No}                     \label{li_updifferentstart}
\li         \ElseNoIf
\li       $u \gets u$ with first letter deleted
\li      $v \gets v$ with first letter deleted
\li      \If $p \neq \epsilon$
\li          \Then $p \gets p$ with first letter deleted
         \End
\li      \Return $\proc{WP-Prefix}(u,v,p)$   \label{li_rec_nomop}
\End \label{li_end_b}

\li \ElseNoIf
\li $\kw{let}\  X, Y, u'$ be such that $u = XY u'$ \label{li_start_c}

\li \If $p$ is a prefix of neither $X$ nor $\ol{X}$
\li \Then \Return \const{No} \label{li_pnotprefix}

\li \ElseIf $v$ does not begin either with $XY$ or with $\ol{XY}$
\li \Then \Return \const{No} \label{li_vstartswrong}

\li \ElseIf $u = XYZ u''$ and $v = XYZ v''$
\li    \Then \If $u''$ is $Z$-active
\li       \Then \Return $\proc{WP-Prefix}(Z u'', Z v'', \epsilon)$ \label{li_rec_case1a}
\li       \Else \Return $\proc{WP-Prefix}(\ol{Z} u'', \ol{Z} v'', \epsilon)$ \label{li_rec_case1b}
       \End

\li \ElseIf $u = XY u'$ and $v = XY v'$
\li     \Then \If $p$ is a prefix of $X$
\li         \Then \Return $\proc{WP-Prefix}(u',v', \epsilon)$ \label{li_rec_case2a}
\li         \Else \Return $\proc{WP-Prefix}(u',v', Z)$ \label{li_rec_case2b}
        \End

\li \ElseIf $u = XYZ u''$ and $v = \ol{XYZ} v''$
\li     \Then \If $u''$ is $Z$-active
\li         \Then \Return $\proc{WP-Prefix}(Z u'', Z v'', \epsilon)$ \label{li_rec_case3a}
\li         \Else \Return $\proc{WP-Prefix}(\ol{Z} u'', \ol{Z} v'', \epsilon)$ \label{li_rec_case3b}
        \End

\li \ElseIf $u = XY u'$ and $v = \ol{XYZ} v''$
\li     \Then \Return $\proc{WP-Prefix}(u', Z v'', \epsilon)$ \label{li_rec_case4}

\li \ElseIf $u = XYZ u''$ and $v = \ol{XY} v'$
\li     \Then \Return $\proc{WP-Prefix}(\ol{Z} u'', v', \epsilon)$ \label{li_rec_case5}

\li \ElseIf $u = XY u'$ and $v = \ol{XY} v'$
\li     \Then \kw{let} $z$ be the maximal common suffix of $Z$ and $\ol{Z}$
\li           \kw{let} $z_1$ be such that $Z = z_1 z$
\li           \kw{let} $z_2$ be such that $\ol{Z} = z_2 z$
\li           \If $u'$ does not begin with $z_1$ or $v'$ does not begin with $z_2$;
\li               \Then \Return \const{NO} \label{li_case6no}
\li               \Else \kw{let} $u''$ be such that $u' := z_1 u''$
\li                     \kw{let} $v''$ be such that $v' := z_2 v''$;
\li                     \Return $\proc{WP-Prefix}(u'', v'', z)$ \label{li_rec_case6} \label{li_end_c}
              \End
        \End
    \End
\end{codebox}
\caption{Algorithm for the word problem of a $C(4)$ presentation}
\end{figure}
It
takes as input a piece of the presentation $p \in A^*$ and two words
$u, v \in A^*$ and outputs \textbf{YES} if $u \equiv v$ and $p$ is a
possible prefix of $u$ (and hence also of $v$). Otherwise it outputs
\textbf{NO}. In particular, if
$p = \epsilon$ then the algorithm outputs \textbf{YES} if $u \equiv v$ and
\textbf{NO} if $u \not\equiv v$, thus solving the word problem for the
presentation. See \cite[Lemma~5]{K_smallover1} and
\cite[Lemma~6]{K_smallover1} for proofs of correctness and termination
respectively.

The proof strategy for our main result is to show that this algorithm
can be implemented on a deterministic 2-tape prefix-rewriting automaton
with bounded expansion. The results of Section~\ref{sec_prefmod} then
allow us to conclude that the word problem is a deterministic rational relation.

\begin{theorem}\label{thm_main}
Let $\langle A \mid R \rangle$ be a finite monoid presentation satisfying
the small overlap condition $C(4)$. Then the relation
$$\lbrace (u, v) \in A^* \times A^* \mid u \equiv v \rbrace$$
is deterministic rational and reverse deterministic rational. Moreover,
one can, starting from the presentation, effectively compute 2-tape
deterministic automata recognising this relation and its reverse.
\end{theorem}
\begin{proof}
Let $k$ be twice the maximum length of a relation word in the presentation.
We construct a deterministic 2-tape $k$-prefix-rewriting automaton recognising
the desired relation, and an expansion bound for this automaton.
By Theorem~\ref{thm_bpm_imp_rational}, this suffices to show that the
given relation is deterministic rational and that a 2-tape deterministic
automaton for it can be effectively constructed. Since the $C(4)$ condition
on the presentation is entirely left-right symmetric, the claim regarding
the reverse relation also follows.

Let $P$ be the set of all pieces of the presentation $\langle A \mid R \rangle$,
and let $+$ be a new symbol not in $P$. Recall that $\epsilon$
is by definition a piece of every presentation, so certainly $\epsilon \in P$.
Let $W = A^k \cup A^{< k} \$ $.
We define a 2-tape prefix-rewriting automaton with
\begin{itemize}
\item state set $P \cup \lbrace + \rbrace$;
\item initial state $\epsilon$,
\item unique terminal state $+$;
\end{itemize}
and edges defined as follows.
\begin{itemize}
\item[(A)] an edge from $\epsilon$ to $+$ labelled $(\$, \$, \$, \$)$.
\item[(B)] for every $u \in W$ with $u \neq \$ $ and such that $u$ has no clean
overlap prefix of the form $XY$, and every $v \in W$ such that $v \neq \$ $ and $u$ and $v$
begin with the same letter, a transition from
$p$ to $p'$ labelled $(u, u', v, v')$ where $u'$, $v'$ and $p'$ are
obtained from $u$, $v$ and $p$ respectively by deleting the first letter.
\end{itemize}

In addition for every $p \in P$ and $u,v \in W$ such that $u$ has a clean
overlap prefix (say $XY$) and $p$ is a prefix of either $X$ or $\ol{X}$ or
both, the automaton may have an edge from $p$ to another state in $P$ as
follows:
\begin{itemize}
\item[(C1)] If $u = XYZ u''$, $v = XYZv''$ and $u''$ is $Z$-active, the automaton
      has an edge from $p$ to $\epsilon$ labelled $(u, Zu'', v, Zv'')$.
\item[(C2)] If $u = XYZ u''$, $v = XYZv''$ and $u''$ is \textbf{not} $Z$-active,
      the automaton has an edge from $p$ to $\epsilon$ labelled $(u, \ol{Z}u'', v, \ol{Z} v'')$.
\item[(C3)] If $u = XY u'$, $v = XY v'$, $u$ and $v$ do not both have $XYZ$ as a 
      prefix, and $p$ is a prefix of $X$, the automaton has an edge from $p$ to $\epsilon$
      labelled $(u, u', v, v')$.
\item[(C4)] If $u = XY u'$, $v = XY v'$, $u$ and $v$ do not both have $XYZ$ as a 
      prefix, and $p$ is \textbf{not} a prefix of $X$, the automaton has an
       edge from $p$ to $Z$ with label $(u, u', v, v')$.
\item[(C5)] If $u = XYZ u''$, $v = \ol{XYZ}v''$ and $u''$ is $Z$-active, the
      automaton has an edge from
      $p$ to $\epsilon$ labelled $(u, Zu'', v, Zv'')$.
\item[(C6)] If $u = XYZ u''$, $v = \ol{XYZ}v''$ and $u''$ is \textbf{not} $Z$-active,
      the automaton has an edge from $p$ to $\epsilon$ labelled $(u, \ol{Z}u'', v, \ol{Z} v'')$.
\item[(C7)] If $u = XY u'$, $v = \ol{XYZ} v''$ and $u$ does not have $XYZ$ as a prefix,
      the automaton has an edge from $p$ to $\epsilon$ labelled $(u, u', v, Z v'')$.
\item[(C8)] If $u = XYZ u''$, $v = \ol{XY} u'$ and $v$ does not have $\ol{XYZ}$ as
      a prefix, the automaton has an edge from $p$ to $\epsilon$ labelled $(u, \ol{Z} u'', v, v')$.
\item[(C9)] If $u = XY u'$, $v = \ol{XY} v'$, $u$ does not begin with $XYZ$, $v$ does
      not begin with $\ol{XYZ}$, $z$ is the maximum common suffix of $Z$ and
      $\ol{Z}$, $Z = z_1 z$, $\ol{Z} = z_2 z$, $u' = z_1 u''$, $v' = z_2 v''$,
      the automaton has an edge from $p$ to $z$ labelled $(u, u'', v, v'')$.
\end{itemize}
First, notice that this automaton is deterministic. Indeed, all edges
leaving a given vertex $p \in P$ have labels of the form $(u, x, v, y)$ with
$u, v \in W$. Notice that no member of the set $W$ is a prefix of another;
it follows that no word has two distinct words in $W$ as prefixes, which
means that the choice of prefixes $u$ and $v$ to act on is uniquely determined
by the configuation in which the action is to be applied.
Now it can be verified by examination that the various conditions on $u$, $v$
and $p$ which result in the inclusion of an edge from $p$ with label of the
form $(u, x, v, y)$ are mutually exclusive, so that there is at most one such
edge, and hence at most one transition applicable in any given configuration.

It is now an entirely routine matter to prove by induction
that for every piece $p \in A^*$ and words $u, v \in A^*$ we have
$$(u \$, v \$, p) \ \to^* \ (\$, \$, +)$$
if and only if the algorithm outputs \textbf{YES}, that is, if and only if
$u \equiv v$ and $p$ is a possible prefix of $u$. Transitions of types
B, C1, C2, C3, C4, C5, C6, C7, C8 and C9 correspond to the recursive calls at
lines 15, 24, 25, 28, 29, 32, 33, 35, 37, 46 respectively, while
transition of type A corresponds to termination with the answer \textbf{YES}
at line 3 of the algorithm. The conditions under which the algorithm terminates
with the answer \textbf{NO} (at lines 4, 7, 9, 19, 21 and 43) all correspond
to non-terminal configurations of the automaton in which no transitions are
applicable. It follows from \cite[Lemma~7]{K_smallover1} that the tests for
clean overlap prefixes and $Z$-activity on the buffer contents are equivalent
to performing the corresponding tests on the whole of the remaining input,
as demanded by the algorithm.

In particular, we have
$$(u \$, v \$, \epsilon) \ \to^* \ (\$, \$, +)$$ if and only if
$u \equiv v$, as required to show that our prefix-rewriting automaton
solves the word problem.
It remains only to find an expansion bound for the automaton. Let
$b$ be the length of the longest relation word in the presentation
$\langle A \mid R \rangle$.

Suppose $(u_0, v_0, q_0) \to^* (u_1, v_1, q_1)$ and suppose that
$u_0 = z_0 u_0'$
and $v_0 = z_0 v_0'$ where $z_0$ is either a proper suffix of a relation word
or the empty word. We
claim that there are factorisations
$u_1 = z_1 u_1'$ and $v_1 = z_1 v_1'$ where $z_1$ is a proper suffix of
relation word or the empty word, $|u_1'| \leq |u_0'|$ and $|v_1'| \leq |v_0'|$. 

We consider first the one-step case, that is, where $(u_0, v_0, q_0) \to (u_1, v_1, q_1)$.
If the transition from $(u_0, v_0, q_0)$ to $(u_1, v_1, q_1)$ is of type A or B then the claim is clear, so suppose the
transition is of type C1-C9. Then from the definitions of these transitions,
we must have $u_0 = XY u'$ for some maximum piece prefix $X$ and middle word
$Y$ of a relation word $XYZ$. Now $XY$ cannot be
a piece, so it cannot be a prefix of $z_0$, which is a proper suffix of a
relation word. Thus, we must have $|XY| > |z_0|$ and hence
$|u'| < |u_0'|$. Looking again at the definitions of the transitions, we
see that $u_1$ and $v_1$ either
\begin{itemize}
\item[(i)] are (not necessarily proper) suffixes of $u'$ and $v'$ respectively;
or
\item[(ii)] have the form $u_1 = Z u''$ and $v_1 = Z v''$ where $u''$
and $v''$ are (not necessarily proper) suffixes of $u'$ and $v'$ respectively; or
\item[(iii)] have the form $u_1 = \ol{Z} u''$ and $v_1 = \ol{Z} v''$ where $u''$
and $v''$ are (not necessarily proper) suffixes of $u'$ and $v'$ respectively.
\end{itemize}
In case (i) it suffices to set $z_1 = \epsilon$ and $u_1' = u_1$.
In case (ii) [respectively, case (iii)] it suffices to set
$z_1 = Z$ [respectively, $z_1 = \ol{Z}$] and $u_1' = u''$, noting that
$Z$ [respectively, $\ol{Z}$] must be a proper suffix of a relation word since
is a maximal piece suffix of $XYZ$ [$\ol{XYZ}$] and no relation word can
be a piece.

It now follows easily by induction that the claim also holds when
$$(u_0, v_0, q_0) \to^* (u_1, v_1, q_1).$$
In particular, taking $z_0 = \epsilon$ and $u_0' = u_0$
and then writing $u_1 = z_1 u_1'$ as above we have
$$|u_1| = |z_1| + |u_1'| \ \leq \ |z_1| + |u_0'| \ = \ |z_1| + |u_0| \ \leq \ |u_0| + b$$
and similarly $|v_1| \leq |v_0| + b$, as required to show that the
automaton has expansion bound $b$.
\end{proof}

As an immediate corollary we obtain a corresponding statement for
semigroups.
\begin{corollary}\label{cor_mainsemi}
Let $\langle A \mid R \rangle$ be a finite semigroup presentation satisfying
the small overlap condition $C(4)$. Then the relation
$$\lbrace (u, v) \in A^+ \times A^+ \mid u \equiv v \rbrace$$
is deterministic rational and reverse deterministic rational. Moreover,
one can, starting from the presentation, effectively compute 2-tape
deterministic automata recognising this relation and its reverse.
\end{corollary}
\begin{proof}
Since the presentation has no empty relation words, the semigroup with
presentation $\langle A \mid R \rangle$ arises as the
subsemigroup of non-identity elements in the monoid with presentation
$\langle A \mid R \rangle$. It follows that
$$\lbrace (u, v) \in A^+ \times A^+ \mid u \equiv v \rbrace
\ = \ \lbrace (u, v) \in A^* \times A^* \mid u \equiv v \rbrace \setminus \lbrace (\epsilon, \epsilon) \rbrace.$$
Now it is easy to verify that a relation $R$ between free
monoids is a deterministic rational relation only if
$R \setminus \lbrace ( \epsilon, \epsilon) \rbrace$ is a deterministic
rational relation between free semigroups, so the result follows from
Theorem~\ref{thm_main}.
\end{proof}

\section{Consequences}\label{sec_consequences}

In this section we consider a number of interesting consequences and
corollaries of Theorem~\ref{thm_main}. We begin with some terminology
from language theory.

Let $A$ be a finite alphabet, and choose some arbitrary total order $\leq$
on the letters of $A$. Recall that the corresponding \textit{lexicographic order}
is an extension of this order to a total order $\leq_L$ on the free monoid $A^*$,
defined inductively by $\epsilon \leq_L w$ for all $w$, and for all
$x, y \in A$ and $u, v \in A^*$
we have $xu \leq_L yv$ if either $x \neq y$ and $x \leq y$, or $x = y$ and $u \leq_L v$.
Lexicographic order is a total order but not (unless $|A|=1$) a well-order,
since it contains infinite descending chains such as
$$b, \ ab, \ aab, \ aaab, \ \dots, \ a^i b, \ \dots$$
Hence, if $R$ is an equivalence relation on $A^*$ (even a rational one) there
is no guarantee that every equivalence class of $R$ will contain a lexicographically minimal
element. In the case that $R$ is \textit{locally finite} (that is, each equivalence
class is finite), however, every class
must clearly contain a unique lexicographically minimal element, and the
set of elements which are minimal in their class forms a \textit{cross-section}
of the relation, that is, a language of unique representatives for the
equivalence classes of the relation; we shall call these representatives
\textit{lexicographic normal forms}. Remmers showed that if
$\langle A \mid R \rangle$ is a $C(3)$ monoid [semigroup] presentation then the
corresponding equivalence relation on $A^*$ [respectively, $A^+$] is locally
finite
\cite{Higgins92,Remmers71}; it follows that every element of a $C(3)$ monoid
has a lexicographic normal form. Johnson \cite{Johnson85,Johnson86} showed that if $R$ is a deterministic
rational locally finite equivalence relation then the function
which maps each word to the corresponding lexicographic normal form can be
computed by a deterministic transducer. Thus, we obtain the following corollary
to Theorem~\ref{thm_main}.
\begin{corollary}
Let $\langle A \mid R \rangle$ be a monoid presentation satisfying $C(4)$
and suppose $A$ is equipped with a total order.
Then the relation
$$\lbrace (u, v) \in A^* \times A^* \mid u \equiv v \text{ and } v \text{ is a lexicographic normal form} \rbrace$$
is a deterministic rational function.
\end{corollary}

The image of a rational function is always a regular language
\cite[Corollary~II.4.2]{Berstel79}) and deterministic rational functions
can be computed in linear time Johnson \cite[Theorem~5.1]{Johnson86} so
we have:
\begin{corollary}\label{cor_lexmin}
Let $\langle A \mid R \rangle$ be a monoid presentation satisfying $C(4)$
and suppose $A$ is equipped with a total order. Then the lexicographic
normal forms comprise a regular language of unique representatives for elements
of the monoid. Moreover,
there is an algorithm which, given a word $w$ in $A^*$, computes in linear
time the corresponding lexicographic normal form.
\end{corollary}

A monoid $M$ is called \textit{rational} \cite{Sakarovitch87,Sakarovitch90}
if there exists a finite generating set $A$ for $M$ and a regular cross-section
$L \subseteq A^*$ for $M$ such that the normal forms in $L$ are computed
by a transducer.
\begin{corollary}
Every monoid admitting a $C(4)$ presentation is \textit{rational}.
\end{corollary}

Recall that the \textit{rational subsets} of a monoid $M$ are those which
can be obtained from finite subsets by the operations of union, product
and submonoid generation (the ``Kleene star'' operation). If $M$ is generated
by a finite subset $A$ then the rational subsets of $M$ are exactly the
images in $M$ of regular languages over $A$, which means they have natural
finite representations as finite automata over $A$. The
\textit{recognisable subsets} of $M$ are the homomorphic pre-images in $M$
of subsets of finite monoids. In the case that $M$ is a free monoid, the
rational subsets are just the regular languages. Kleene's Theorem asserts
that the rational subsets of a free monoid (that is, the regular languages)
coincide with the recognisable subsets \cite{Hopcroft69}. More generally,
a monoid in which the rational and recognisable subsets coincide is called
a \textit{Kleene monoid}, or sometimes is said to \textit{satisfy Kleene's Theorem}.
Rational monoids were originally introduced in an attempt to obtain a
concrete characterisation of Kleene monoids \cite{Sakarovitch87}, and
indeed every rational monoid is a Kleene monoid (although it transpires
that the converse does not hold). Thus, we obtain:
\begin{corollary}[Kleene's Theorem for Small Overlap Monoids]
Let $M$ be a monoid or semigroup admitting a $C(4)$ presentation, and $S$
a subset of $M$. Then $S$ is rational if and only if $S$ is recognisable.
\end{corollary}

Recall that a collection of subsets of some given base set is called a
\textit{boolean algebra} if it contains the empty set and is closed under
union, intersection and complement. As another corollary of the
rationality of $M$ we obtain the following fact about rational subsets of
$M$.
\begin{corollary}\label{cor_booleanalgebra}
Let $M$ be a monoid admitting a $C(4)$ presentation $\langle A \mid R \rangle$.
Then the rational subsets of $M$ form a boolean algebra. Moreover, if rational
subsets of $M$ are represented by automata over $A$, then the operations of
union, intersection and complement are effectively computable.
\end{corollary}
\begin{proof}
Let $\sigma : A^* \to M$ be the canonical
morphism mapping $A^*$ onto $M$, and let
$$\rho = \lbrace (u, v) \in A^* \times A^* \mid u \equiv v \text{ and } v \text{ is a lexicographic normal form} \rbrace.$$
Suppose $X, Y \in A^*$ are rational subsets,
with say $X = \hat{X} \sigma$ and $Y = \hat{Y} \sigma$ where
$\hat{X}, \hat{Y} \subseteq A^*$ are regular languages.
Then using the facts that $A^* \rho$ contains a unique representative for
every element and that $\rho \sigma = \sigma$, it is readily verified that
$M \setminus X = (A^* \rho \setminus \hat{X} \rho) \sigma$,
$X \cap Y = (\hat{X} \rho \cap \hat{Y} \rho) \sigma$ and $X \cup Y = (\hat{X} \rho \cup \hat{Y} \rho) \sigma$.
The result now follows from the fact that regular languages in a free monoid
form a boolean algebra with effectively computable operations.
\end{proof}

Recall that the \textit{rational subset membership problem} for a finitely
generated monoid $M$ is the problem
of deciding, given a rational subset of $M$ (represented by a finite
automaton over some fixed generating set for $M$) and an
element of $M$ (represented as a word over the same generating set), whether
the given element belongs to the given subset. The
decidability of this problem is independent of the chosen generating set
\cite[Corollary~3.4]{KambitesGraphRat}.
\begin{corollary}
Any monoid admitting a $C(4)$ presentation has decidable rational subset
membership problem (and hence decidable submonoid membership problem). 
\end{corollary}
\begin{proof}
Suppose $M$ has $C(4)$ presentation $\langle A \mid R \rangle$, and let
$\sigma : A^* \to M$ be once again the canonical morphism.
Suppose we are given a finite automaton recognising a language $\hat{X} \subseteq A^*$
(representing the rational subset $\hat{X} \sigma \subseteq M$) and a 
$w \in A^*$ (representing the element $w \sigma \in M$). Certainly we
can compute from the latter a finite automaton recognising the singleton language
$\lbrace w \rbrace$. Hence, by Corollary~\ref{cor_booleanalgebra} we
can compute a finite automaton recognising a language $\hat{Y} \subseteq A^*$
such that $\hat{Y} \sigma = \hat{X} \sigma \cap \lbrace w \rbrace \sigma$. But
$w \sigma \in \hat{X} \sigma$ if and only if
$\hat{X} \sigma \cap \lbrace w \rbrace \sigma$ is non-empty, so this
reduces the problem to deciding emptiness of the regular language $\hat{Y}$;
the latter is well known to be decidable.
\end{proof}

A monoid $M$ is called \textit{asynchronous automatic} (see, for example,
\cite{Hoffmann01}) if there exists a finite generating set $A$ and a regular
language $L \subseteq A^*$ such that $L$ contains a representative for every
element of $M$, and the relation
$$\lbrace (u,v) \in A^* \times A^* \mid ua \equiv v \rbrace$$
is a rational transduction for each $a \in A$ and for $a = \epsilon$.
It has been shown \cite[Theorem~6.2]{Hoffmann01} that rational monoids are
asynchronous automatic, so we also obtain the following.
\begin{corollary}
Every monoid admitting a $C(4)$ presentation is asynchronous automatic.
\end{corollary}

We have already remarked that small overlap conditions are the natural
semigroup-theoretic analogue of the small cancellation conditions extensively
used in combinatorial group theory (see, for example, \cite{Lyndon77}). It
is well known that a group admitting a finite presentation satisfying sufficiently
strong small cancellation conditions is \textit{word hyperbolic} in the sense
of Gromov \cite{Gromov87}. The usual geometric definition of a word hyperbolic
group has no obvious counterpart for more general monoids or semigroups;
however, Gilman \cite{Gilman02} has given a language-theoretic characterisation of
word hyperbolic groups. Specifically, he showed that a group is word
hyperbolic if and only if it admits a finite generating set $A$ and a regular
language $L \subseteq A^*$
containing a representative for every element of $M$ such that the
\textit{multiplication table}
$$\lbrace u \# v \# w^R \mid uv \equiv w \rbrace$$
is a context-free language, where $\#$ is a new symbol not in $A$.
Motivated by this result, Duncan
and Gilman \cite{Duncan04} have suggested calling a monoid \textit{word hyperbolic} if it
satisfies this language-theoretic condition. Since every rational monoid
is word hyperbolic \cite[Theorem~6.3]{Hoffmann01} we can deduce that every
$C(4)$ monoid is word hyperbolic in this sense.
\begin{corollary}
Every monoid admitting a $C(4)$ presentation is word hyperbolic in the sense
of Duncan and Gilman (and furthermore admits a hyperbolic structure with
unique representatives).
\end{corollary}

\section*{Acknowledgements}

This research was supported by an RCUK Academic Fellowship. The author
would like to thank A.~V.~Borovik and V.~N.~Remeslennikov for a number of
helpful conversations.

\bibliographystyle{plain}

\begin{thebibliography}{10}

\bibitem{Berstel79}
J.~Berstel.
\newblock {\em Transductions and Context-Free Languages}.
\newblock Informatik. Teubner, 1979.

\bibitem{Buntrock98}
G.~Buntrock and F.~Otto.
\newblock Growing context-sensitive languages and {C}hurch-{R}osser languages.
\newblock {\em Inform. and Comput.}, 141(1):1--36, 1998.

\bibitem{Duncan04}
A.~Duncan and R.~H. Gilman.
\newblock Word hyperbolic semigroups.
\newblock {\em Math. Proc. Cambridge Philos. Soc.}, 136(3):513--524, 2004.

\bibitem{Elgot65}
C.~C. Elgot and J.~E. Mezei.
\newblock On relations defined by generalized finite automata.
\newblock {\em IBM J. Res. Develop}, 9:47--68, 1965.

\bibitem{Fischer68}
P.~C. Fischer and A.~L. Rosenberg.
\newblock Multitape one-way nonwriting automata.
\newblock {\em J. Comput. System Sci.}, 2:88--101, 1968.

\bibitem{Gilman02}
R.~H. Gilman.
\newblock On the definition of word hyperbolic groups.
\newblock {\em Math. Z.}, 242(3):529--541, 2002.

\bibitem{Gromov87}
M.~Gromov.
\newblock Hyperbolic groups.
\newblock In {\em Essays in Group Theory}, volume~8 of {\em Math. Sci. Res.
  Inst. Publ.}, pages 75--263. Springer, New York, 1987.

\bibitem{Higgins92}
P.~M. Higgins.
\newblock {\em Techniques of semigroup theory}.
\newblock Oxford Science Publications. The Clarendon Press Oxford University
  Press, New York, 1992.
\newblock With a foreword by G. B. Preston.

\bibitem{Hoffmann01}
M.~Hoffmann, D.~Kuske, F.~Otto, and R.~M. Thomas.
\newblock Some relatives of automatic and hyperbolic groups.
\newblock In G.~M.~S. Gomes, J.-E. Pin, and P.~V. Silva, editors, {\em
  Semigroups, Algorithms, Automata and Languages}, pages 379--406, 2003.

\bibitem{Hopcroft69}
J.~E. Hopcroft and J.~D. Ullman.
\newblock {\em Formal Languages and their Relation to Automata}.
\newblock Addison-Wesley, 1969.

\bibitem{Johnson85}
J.~H. Johnson.
\newblock Do rational equivalence relations have regular cross sections?
\newblock In {\em Automata, languages and programming (Nafplion, 1985)}, volume
  194 of {\em Lecture Notes in Comput. Sci.}, pages 300--309. Springer, Berlin,
  1985.

\bibitem{Johnson86}
J.~H. Johnson.
\newblock Rational equivalence relations.
\newblock {\em Theoret. Comput. Sci.}, 47(1):39--60, 1986.

\bibitem{K_smallover1}
M.~Kambites.
\newblock Small overlap monoids {I}: the word problem.
\newblock {\em J. Algebra (to appear)}, 2007.
\newblock (Preprint available at {\tt arXiv:0712.0250 [math.RA]}).

\bibitem{KambitesGraphRat}
M.~Kambites, P.~V. Silva, and B.~Steinberg.
\newblock On the rational subset problem for groups.
\newblock {\em J. Algebra}, 309:622--639, 2007.

\bibitem{Lyndon77}
R.~C. Lyndon and P.~E. Schupp.
\newblock {\em Combinatorial Group Theory}.
\newblock Springer-Verlag, 1977.

\bibitem{Sakarovitch90}
M.~Pelletier and J.~Sakarovitch.
\newblock Easy multiplications {I}{I}. {E}xtensions of rational semigroups.
\newblock {\em Inform. and Comput.}, 88:18--59, 1990.

\bibitem{Remmers71}
J.~H. Remmers.
\newblock {\em Some algorithmic problems for semigroups: a geometric approach}.
\newblock PhD thesis, University of Michigan, 1971.

\bibitem{Remmers80}
J.~H. Remmers.
\newblock On the geometry of semigroup presentations.
\newblock {\em Adv. in Math.}, 36(3):283--296, 1980.

\bibitem{Sakarovitch87}
J.~Sakarovitch.
\newblock Easy multiplications {I}. {T}he realm of {K}leene's theorem.
\newblock {\em Inform. and Comput.}, 74:173--197, 1987.

\end{thebibliography}

\def\cprime{$'$} \def\cprime{$'$} \def\cprime{$'$}

\end{document}